\documentclass[10pt]{amsart}
\usepackage{latexsym,amsmath,amscd,amssymb,epsfig}

\newcommand{\FF} {{\mathbb F}}
\newcommand{\RR} {{\mathbb R}}

\newcommand{\CC} {{\mathbb C}}

\newcommand{\QQ} {{\mathbb Q}}

\newcommand{\ZZ} {{\mathbb Z}}
\newcommand{\SSS}{{\mathbb S}}
\newcommand{\BBB}{{\mathbb B}}

\newcommand{\FFF}{{\mathcal F}}

\newcommand{\symm} {{\mathfrak S}}

\newcommand{\sgn} {{\mathop{\rm sgn}\nolimits}}
\newcommand{\rank} {{\mathop{\rm rank}\nolimits}}
\newcommand{\str} {{\mathop{\rm star}\nolimits}}
\newcommand{\Bd} {{\mathop{\rm Bd}\nolimits}}

\newcommand{\coker}{{\mathop{\rm coker}\nolimits}}

\newcommand{\comment}[1]{}

\newtheorem{theorem}{Theorem}
\newtheorem{definition}[theorem]{Definition}
\newtheorem{corollary}[theorem]{Corollary}
\newtheorem{proposition}[theorem]{Proposition}

\newtheorem{question}[theorem]{Question}

\newtheorem{remark}[theorem]{Remark}
\newtheorem{example}[theorem]{Example}

\begin{document}
\title{The cyclotomic polynomial topologically}

\author{Gregg Musiker}
\address{Gregg Musiker \\
School of Mathematics\\
University of Minnesota\\
Minneapolis, MN 55455}
\email{musiker@math.umn.edu}

\author{Victor Reiner}
\address{Victor Reiner \\
School of Mathematics\\
University of Minnesota\\
Minneapolis, MN 55455}
\email{reiner@math.umn.edu}


\thanks{First author supported by NSF Postdoctoral Fellowship and NSF grant DMS-1067183.
Second author supported by NSF grant DMS-1001933.}

\dedicatory{To the memory of Mark Feshbach.}

\begin{abstract}
We interpret the coefficients of the cyclotomic polynomial
in terms of simplicial homology.
\end{abstract}

\maketitle

\section{Introduction}
\label{sec:introduction}

This paper studies the cyclotomic polynomial $\Phi_n(x)$, which is
defined as the minimal polynomial over $\QQ$ 
for any primitive $n^{th}$ root of unity $\zeta$ in $\CC$.
It is monic, irreducible, and has degree  
given by the Euler phi function $\phi(n)$, with formula 
$$
\Phi_n(x) = \prod_{j \in (\ZZ/n\ZZ)^\times}
                     (x-\zeta^j).
$$
The equation
\begin{equation}
\label{eqn:recurrence-precursor}
x^n-1 = \prod_{d | n} \Phi_d(x)
\end{equation}
gives a recurrence showing that all coefficients of $\Phi_n(x)$ lie in $\ZZ$.  

Although well-studied, the coefficients of $\Phi_n(x)$ are mysterious \cite{Bachman, Elder, Moree, Kaplan, Kaz, Lenstra, Zhao}.  We offer here two interpretations for their magnitudes,
as orders of cyclic groups.
In the initial interpretation (Corollary~\ref{cor:first-interpretation} below)
this group is a quotient of the free abelian group $\ZZ[\zeta]$
by a certain full rank sublattice.

The second interpretation is topological,
given by Theorem~\ref{thm:second-interpretation} below,
as the torsion in the homology of a certain simplicial complex associated with a
squarefree integer $n=p_1 \cdots p_d$.
These simplicial complexes originally arose in the work of Bolker \cite{Bolker}, 
reappeared in the work of Kalai \cite{Kalai} and Adin \cite{Adin} 
on higher-dimensional matrix-tree theorems, and were shown to be 
connected with cyclotomic extensions in work of J. Martin and the 
second author \cite{MartinR}.  We review these simplicial
complexes briefly here in order to state
the result;  see Section \ref{sec:rooting} for more details.

Given a positive integer $p$, let $K_p$ denote a $0$-dimensional abstract
simplicial complex having $p$ vertices\footnote{Note that here $K_p$ does 
{\it not} refer to a complete graph on $p$ vertices;  we hope that this
causes no confusion.}, which we will label by
the residues 
$$
\{ 0\bmod{p}, \,\, 1\bmod{p}, \,\, \ldots, \,\, (p-1)\bmod{p}\}
$$ 
for reasons that will become clear in a moment.

Given primes $p_1,\ldots,p_d$, let
$$
K_{p_1,\ldots,p_d}:=K_{p_1} * \cdots * K_{p_d}
$$ 
be the {\it simplicial join}, \cite[\S 62]{Munkres}, 
of $K_{p_1}, \ldots, K_{p_d}$.  This is a 
pure $(d-1)$-dimensional abstract simplicial complex, that
may be thought of as the {\it complete $d$-partite 
complex} on vertex sets $K_{p_1}$ through $K_{p_d}$ of sizes $p_1, \dots, p_d$.

The {\it facets} (maximal simplices) of $K_{p_1, \ldots, p_d}$ are labelled
by sequences of residues $(j_1\bmod{p_1}, \ldots, j_d\bmod{p_d})$.  Denoting 
the squarefree product $p_1 \cdots p_d$ by $n$, the
Chinese Remainder Theorem isomorphism
\begin{equation}
\label{eqn:Chinese-remainder}
\ZZ/p_1\ZZ \,\, \times \cdots \times \,\, \ZZ/p_d\ZZ \,\, 
\overset{\Xi}{\longrightarrow} 
\,\, \ZZ/n\ZZ
\end{equation}
allows one to label such a facet by a residue $j\bmod{n}$;  call this facet
$F_{j\bmod{n}}$.  Then for any subset $A \subseteq \{0,1,\ldots,\phi(n)\}$,
let $K_A$ denote the subcomplex of $K_{p_1,\ldots,p_d}$ which is generated
by the facets $\{ F_{j\bmod{n}} \}$ as $j$ runs through the following set of
residues:
$$
A \cup \{\phi(n)+1, \,\, \phi(n)+2, \,\, \ldots, \,\, n-2, \,\, n-1\}.
$$

Our first main result interprets the magnitudes of the 
coefficients of $\Phi_n(x)$.  Let $\tilde{H}_i(-;\ZZ)$ denote reduced simplicial homology with coefficients in $\ZZ$.

\begin{theorem}
\label{thm:second-interpretation}
For a squarefree positive integer $n=p_1 \cdots p_d$, with
cyclotomic polynomial $\Phi_n(x) = \sum_{j=0}^{\phi(n)} c_j x^j$,
one has
$$
\tilde{H}_i(K_{\{j\}} ; \ZZ) =
\begin{cases}
\ZZ/c_j\ZZ &\text{ if } i=d-2, \\
\ZZ        &\text{ if both }i=d-1 \text{ and }c_j =0,\\
0          &\text{ otherwise.}
\end{cases}
$$
\end{theorem}

\noindent
It should be noted that, starting with $d =3$ and for large enough primes $p_i$, one can exhibit 
$(d-1)$-dimensional subcomplexes of $K_{p_1,\ldots,p_d}$ with the following properties:  
like $K_{\{j\}}$, their only nonvanishing homology group lies in dimension $d-2$ and consists 
entirely of torsion, but unlike $K_{\{j\}}$, 
this torsion group need not be cyclic and can require arbitrarily many generators.

We furthermore interpret topologically the signs of the coefficients in $\Phi_n(x)$.
For this, we use oriented simplicial homology, and orient the facet $F_{j\bmod{n}}$ 
having $j \equiv j_i \bmod{p_i}$ for 
$i=1,2,\ldots,d$ as
\begin{equation}
\label{eqn:oriented-facet}
[F_j] = [F_{j\bmod{n}}]=[j_1\bmod{p_1}, \,\, \ldots, \,\, j_d\bmod{p_d}].
\end{equation}

\begin{theorem}
\label{thm:sign-interpretation}
Fix a squarefree positive integer $n=p_1 \cdots p_d$ with 
cyclotomic polynomial $\Phi_n(x) = \sum_{j=0}^{\phi(n)} c_j x^j$.
Then for any $j \neq j'$ such that $c_j,c_{j'} \neq 0$, 
one has 
$
\tilde{H}_{d-1}(K_{\{j,j'\}};\ZZ) \cong \ZZ,
$
and any nonzero $(d-1)$-cycle $z=\sum_{\ell} b_\ell [F_\ell]$ 
in this homology group 
will have $b_{j}, b_{j'} \neq 0$, with
$$
\frac{c_j}{c_{j'}} = - \frac{b_{j'}}{b_{j}}.
$$
In particular, $c_j,c_{j'}$ have the same sign
if and only if $b_j, b_{j'}$ have opposite signs.
\end{theorem}

\begin{example} \rm
\label{ex:n15} 
We illustrate Theorems~\ref{thm:second-interpretation} and ~\ref{thm:sign-interpretation} 
for $n=15$.  Here $d=2,
p_1=3,p_2=5$, and $\phi(n)=2\cdot4=8$.
The cyclotomic polynomial is
$$
\begin{aligned}
\Phi_{15}(x) &= 1 - x + x^3 - x^4 +x^5 - x^7+x^8 \\
            &=(+1) \cdot( x^0 +x^3+x^5+x^8) + (-1) \cdot (x^1 + x^4 +x^7 ) + 0 \cdot (x^2 +x^6).
\end{aligned}
$$

\begin{figure}
\epsfxsize=125mm
\epsfbox{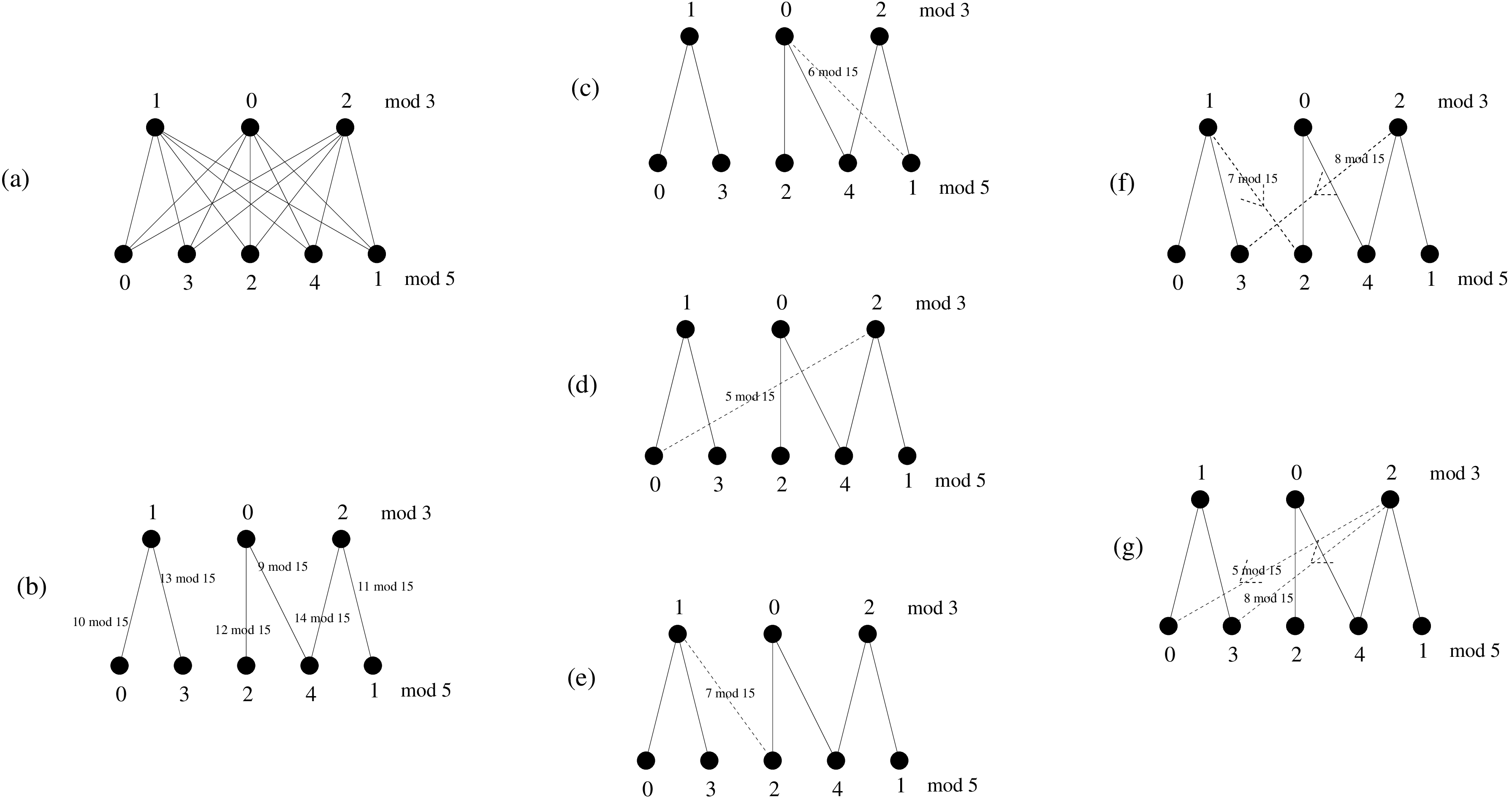}
\caption{The case of $\Phi_{15}(x)$}
\label{intro-example-figure}
\end{figure}

\noindent
The complex $K_{p_1,p_2}=K_{3,5}$ is a complete bipartite graph with vertex sets
labelled as in Figure 1(a) .  The subcomplex $K_\varnothing$ generated
by the edges $F_{j\bmod{15}}$ with $j \in \{\phi(n)+1,\phi(n)+2,\ldots,n-1\}=\{9,10,11,12,13,14\}$
is the subgraph shown in Figure 1(b).  

To see why the coefficient $c_6=0$ in $\Phi_{15}(x)$, one adds the 
edge $F_{6\bmod{15}}$ to the graph  $K_\varnothing$, obtaining the
graph $K_{\{6\}}$, shown in Figure 1(c), 
which has 
$$
\begin{aligned}
\tilde{H}_0(K_{\{6\}};\ZZ) &= \ZZ = \ZZ/0\ZZ \\
\tilde{H}_1(K_{\{6\}};\ZZ) &= \ZZ.
\end{aligned}
$$

To see why the coefficients $c_5=+1$ or $c_7=-1$
have magnitude $1$, one adds the 
edge $F_{5\bmod{15}}$ or $F_{7\bmod{15}}$ to the graph  $K_\varnothing$, obtaining the
graphs $K_{\{5\}}$ or $K_{\{7\}}$ shown in Figures 1(d) and 1(e),
which have
$$
\begin{array}{rcccl}
\tilde{H}_0(K_{\{5\}};\ZZ) & = & 0 & = & \ZZ/(+1)\ZZ \\
\tilde{H}_0(K_{\{7\}};\ZZ) & = & 0 & = & \ZZ/(-1)\ZZ.
\end{array}
$$

To understand the signs of the coefficients, note first that,
by convention, $\Phi_{15}(x)$ is monic, so the coefficient $c_8=c_{\phi(n)}=+1$.
Therefore any other coefficient $c_j$ should have sign
$$
\sgn(c_j) = \frac{\sgn(c_j)}{\sgn(c_8)} = - \frac{\sgn(b_8)}{\sgn(b_j)}
$$
where $z=\sum_{i} b_i [F_i]$ is a nontrivial cycle in $K_{\{j,8\}}$, in which
the edge $[F_j]$ is directed from the vertex 
$(j_1\bmod{3})$ toward the vertex $(j_2\bmod{5})$.  As shown
in Figures 1(f) and 1(g), the nontrivial cycle in $K_{\{7,8\}}$ has $[F_7],[F_8]$ oriented
in the {\it same} direction, explaining why $c_7=-1$, while the nontrivial cycle 
in $K_{\{5,8\}}$ has $[F_5],[F_8]$ oriented
in the {\it opposite} direction, explaining why $c_5=+1$.
\end{example}

  The remainder of the paper is structured as follows.  
Section~\ref{sec:first-interpretation} describes an initial interpretation
for the cyclotomic polynomial, which applies much more generally to any
monic polynomial in $\ZZ[x]$.  Section~\ref{sec:duality} reviews some
facts, underlying the main results,
about duality of matroids, Pl\"ucker coordinates, and oriented matroids.  
Section~\ref{sec:rooting} recalls results and 
establishes terminology on Kalai's higher dimensional 
spanning trees in a simplicial complex.  
Section~\ref{sec:d-partite} discusses further properties of the 
simplicial complex $K_{p_1,\ldots,p_d}$ whose subcomplexes appear in 
Theorem~\ref{thm:second-interpretation} and \ref{thm:sign-interpretation}.
Section~\ref{sec:all-together} prove these theorems.
Section~\ref{sec:attachment} gives some reformulations of
Theorem~\ref{thm:second-interpretation} and \ref{thm:sign-interpretation}
suggested to the authors by D. Fuchs.
Section~\ref{sec:concordance} explains how well-known properties of
the cyclotomic polynomial manifest themselves topologically.

We remark that, since the submission of this paper for publication, an alternate proof of Theorem~\ref{thm:second-interpretation} has appeared in the work of R. Meshulam \cite{Meshulam}, using the {\it Fourier transform} on the group $\ZZ_{p_1} \times \cdots \times \ZZ_{p_d}$.

\section{Coefficients of monic polynomials in $\ZZ[x]$}
\label{sec:first-interpretation}

Our goal here is an initial interpretation for the coefficients of
$\Phi_n(x)$, which applies more generally to the coefficients of
{\it any} monic polynomial $f(x)$ in $\ZZ[x]$.
Recall that when $f(x)$ is of degree $r$, one has an isomorphism of $\ZZ$-modules
$$
\begin{aligned}
\ZZ^r &\longrightarrow &\ZZ[x]/(f(x)) \\
(a_0,a_1,\ldots,a_{r-1})&\longmapsto & \sum_{j=0}^{r-1} a_j \overline{x}^j.
\end{aligned}
$$
As notation, let $R:=\ZZ[x]/(f(x))$, and
for a subset $A$ of some abelian group,
let $\ZZ A$ denote the collection of all $\ZZ$-linear 
combinations of elements of $A$.  

\begin{proposition}
\label{prop:monic-coefficient-interpretation}
For a monic polynomial $f(x)=\sum_{j=0}^r c_j x^j$ of degree $r$ in $\ZZ[x]$, 
one has an isomorphism of abelian groups
$$
R /  \ZZ A \quad  \cong  \quad \ZZ/c_j\ZZ
$$
where  
$A$ is the subset of size $r$ given as $\{\overline{1},\overline{x},\overline{x}^2,\ldots,\overline{x}^r \} \setminus \{\overline{x}^j\}$.
\end{proposition}
\begin{proof}
Consider the matrix in $\ZZ^{r \times (r+1)}$
$$
\left[
\begin{matrix}
1 & 0 & 0 & \cdots & 0 & -c_0 \\
0 & 1 & 0 & \cdots & 0 & -c_1 \\
0 & 0 & 1 & \cdots & 0 & -c_2 \\
\vdots & \vdots & \vdots &  & \vdots & \vdots \\
0 & 0 & 0 & \cdots & 1 & -c_{r-1} \\
\end{matrix}
\right]
$$
whose columns express the elements of 
$\{\overline{1},\overline{x},\overline{x}^2,\ldots,\overline{x}^r \}$
uniquely in the $\ZZ$-basis
$\{\overline{1},\overline{x},\overline{x}^2,\ldots,\overline{x}^{r-1} \}$ 
for $R=\ZZ[x]/(f)$.  The $r \times r$ submatrix obtained by restricting this matrix to
the columns indexed by $A$ is equivalent by row and column operations (invertible over $\ZZ$) to a diagonal matrix with diagonal entries $(1,1,\ldots, 1, -c_j)$.
Hence $R/\ZZ A \cong \ZZ/c_j\ZZ$.
\end{proof}

The special case where $f(x)$ is the cyclotomic polynomial $\Phi_n(x)$
leads to the following considerations.
Fix once and for all a primitive
$n^{th}$ root of unity $\zeta$.

\begin{corollary}
\label{cor:first-interpretation}
The cyclotomic polynomial $\Phi_n(x)=\sum_{j=0}^{\phi(n)} c_j x^j$ has
$$
\ZZ[\zeta]/\ZZ A \cong \ZZ/c_j\ZZ
$$
where 
$
A=\{1,\zeta,\zeta^2, \ldots, \zeta^{\phi(n)} \} \setminus \{\zeta^j\}.
$
\end{corollary}
\begin{proof}
Apply the previous proposition with $f(x)=\Phi_n(x)$ and $r=\phi(n)$,
noting that the ring map $\ZZ[x] \rightarrow \ZZ[\zeta]$ sending $x$ to $\zeta$
will also send $x^j$ to $\zeta^j$, and induce an isomorphism $\ZZ[x]/(\Phi_n(x)) \rightarrow \ZZ[\zeta]$.
\end{proof}

For later use (see the proof of Theorem~\ref{thm:all-together}),
we prove here that the set
$$
P_n:=\{ \zeta^m \}_{m \in \left( \ZZ/n\ZZ\right)^\times} 
$$
of all {\it primitive} $n^{th}$ roots of unity within
$\ZZ[\zeta]$ forms a $\ZZ$-basis whenever $n$ is squarefree.
This is a sharpening of an observation
of Johnsen \cite{Johnsen}, who noted that $P_n$ forms
a $\QQ$-basis of $\QQ[\zeta]$ in the same situation.

\begin{proposition}
\label{prop:Johnsen-precision}
When $n$ is squarefree, the collection $P_n$ of all primitive $n^{th}$ roots
of unity forms a $\ZZ$-basis for $\ZZ[\zeta]$.
\end{proposition}
\begin{proof}
Note that $|P_n|=\phi(n)$ 
and that $\ZZ[\zeta]$ is a free abelian group of rank $\phi(n)$.
Therefore it suffices to show that $P_n$ is a
$\ZZ$-spanning set for $\ZZ[\zeta]$.

This spanning is clear when $n=p$ is prime since the only
non-primitive root of unity is 
$1=-\left( \zeta + \zeta^2 + \cdots + \zeta^{p-1} \right)$.

When $n=p_1\cdots p_d$ for $d > 1$, temporarily use the
notation $\zeta_n$ for a fixed primitive $n^{th}$ root of unity,
and similarly for $\zeta_{p_i}$.  The ring inclusions defined by
$$
\begin{aligned}
\ZZ[\zeta_{p_i}] & \overset{f_i}{\hookrightarrow} & \ZZ[\zeta_n] \\ 
\zeta_{p_i}      & \longmapsto & \left( \zeta_n \right)^{\frac{n}{p_i}}
\end{aligned}
$$
assemble to give a $\ZZ$-module map $f:=f_1 \otimes \cdots \otimes f_d$
\begin{equation}
\label{eqn:assembled-map}
\begin{array}{rcl}
\ZZ[\zeta_{p_1}] \otimes_\ZZ \cdots \otimes_\ZZ \ZZ[\zeta_{p_d}]
& \longrightarrow &\ZZ[\zeta_n] \\
& & \\
\zeta_{p_1}^{a_1} \otimes \cdots \otimes \zeta^{a_d}_{p_d}
& \longmapsto &\zeta_n^j \\
\end{array}
\end{equation}
where $j\equiv \sum_{i=1}^p a_i \frac{n}{p_i}\bmod{n}$.
Note that 
$$
\begin{aligned}
\frac{n}{p_i} & \equiv 0\bmod{p_{i_0}} \text{ for }i \neq i_0,\\
\frac{n}{p_{i_0}} &\not\equiv 0\bmod{p_{i_0}}.
\end{aligned}
$$
Consequently, for each $i_0=1,2,\ldots,d$, one
has $j \equiv a_{i_0} \frac{n}{p_{i_0}} \bmod{p_{i_0}}$,
and as $a_{i_0}$ runs over $\ZZ/p_{i_0}\ZZ$,
the product  $a_{i_0} \frac{n}{p_{i_0}}$ does the same.
The map $f$ is then surjective
by the Chinese Remainder Theorem isomorphism
\eqref{eqn:Chinese-remainder}.

Lastly, note that the source of $f$ has
a $\ZZ$-basis consisting of those 
$
\zeta_{p_1}^{a_1} \otimes \cdots \otimes \zeta^{a_d}_{p_d}
$
in which each $a_i \not\equiv 0\bmod{p_i}$.
Since this means that each  $a_i\frac{n}{p_i} \not\equiv 0\bmod{p_i}$,
this basis maps under $f$ to the set 
$P_n$. Hence $P_n$ must $\ZZ$-linearly span the target $\ZZ[\zeta]$,
and furthermore form a $\ZZ$-basis for the target.
\end{proof}

\section{Duality of matroids or Pl\"ucker coordinates}
\label{sec:duality}

We will need a version of the linear algebraic
duality between Pl\"ucker coordinates for
complementary Grassmannians $G(r,\FF^n), G(n-r,\FF^n)$,
or equivalently, the duality between bases and cobases
in coordinatized matroids.

\begin{proposition}
\label{prop:duality}
Let $0 \leq r \leq n$.
Let $M$ and $M^\perp$ be matrices in $\FF^{r \times n}$ and $\FF^{(n-r) \times n}$, 
respectively, both of maximal rank, with the following property: $\ker M$ is equal to
the row space of $M^{\perp}$, or equivalently, $\ker M^{\perp}$ 
is the row space of $M$.  Then
\begin{enumerate}
\item[(i)] there exists a scalar $\alpha$ in $\FF^\times$ having
the following property:  for every $(n-r)$-subset $T$ of $[n]$, with complementary set $T^c$,
$$
\det \left( M\big|_{T^c} \right) 
 = \pm \alpha \cdot \det \left( M^{\perp}\big|_{T} \right)
$$
where $A\big|_J$ denotes the restriction of a matrix $A$ to the
subset of columns indexed by $J$, and where the $\pm$ sign depends upon
the set $T$.
\item[(ii)] if one furthermore assumes that $\FF=\QQ$, that
$M$ and $M^\perp$ have entries in $\ZZ$, and that there exists at least one $(n-r)$-subset
$T_0$ for which $M\big|_{T_0^c}, M^{\perp} \big|_{T_0}$
are both invertible over $\ZZ$, then the scalar $\alpha$ above equals $\pm 1$,
and one has for every $(n-r)$-subset $T$,
$$
\coker \left( M\big|_{T^c} \right) 
\cong  
\coker \left( M^{\perp}\big|_{T} \right).
$$
Here we are regarding $M\big|_{T^c}$ and $M^{\perp}\big|_{T}$ as operators on 
$\ZZ^{r}$ and $\ZZ^{n-r}$, respectively.
\end{enumerate}
\end{proposition}
\begin{proof}
For assertion (i), note that the hypotheses and conclusions
are unchanged if one performs row operations over $\FF$ separately on
$M, M^\perp$, and if one permutes simultaneously the columns of $M, M^\perp$
by the same permutation of $[n]$.  

Thus one can assume without loss of generality that the full rank matrix $M$ has the
form $M=[\,\, I_r \,\, | \,\, A \,\, ]$ for some
matrix $A$ in $\FF^{r \times (n-r)}$.  In this case, the kernel of $M$ is spanned
by the rows of $[\,\, -A^t \,\, | \,\, I_{n-r} \,\, ]$, and hence by performing row operations
on $M^\perp$, one can assume without loss of generality that 
$M^\perp = [ \,\, -A^t \,\, | \,\, I_{n-r} \,\, ]$.  

When $M, M^\perp$ have these special forms, one has that
$$
\begin{array}{ccccccc}
M\big|_{T^c} &= &[& \quad I_r\big|_{\,[r]\setminus T}  \quad 
               & | &  \quad A\big|_{\,[n] \setminus ([r] \cup T)} \quad &] \\
 & & & & & & \\
M^{\perp}\big|_{T} &= &[&\quad -A^t\big|_{\,[r] \cap T}  \quad 
               & | &  \quad I_{n-r}\big|_{T} \quad &].
\end{array} 
$$
Hence $\det \left( M\big|_{T^c} \right)$ will be, up to sign, the
determinant of 
$A$ restricted to 
its columns in $[n] \setminus ([r] \cup T)$,
and to its rows in $[r] \cap T$,
while $\det \left( M^\perp\big|_{T} \right)$ will be, up to sign,
the determinant of
$-A^t$ restricted to its columns in $[r] \cap T$,
and to its rows in $[n] \setminus ([r] \cup T)$.  
These determinants are the same up to sign.

For assertion (ii), again the hypotheses and
conclusion are unchanged if one performs row operations {\it invertible over $\ZZ$}
to $M, M^\perp$, and if one permutes columns simultaneously in $M, M^\perp$.
Thus without loss of generality, one can assume that 
$$
\begin{aligned}
T^c&=[r]\\
T&=[n] \setminus [r]\\
M&=[ \,\, I_r \,\, | \,\, A \,\, ]\\
M^\perp&=[ \,\, -A^t \,\, | \,\, I_{n-r} \,\, ]
\end{aligned}
$$
as above.  But in this case, one can check that 
$
\coker \left( M\big|_{T^c} \right)
\cong 
\coker \left( M^{\perp}\big|_{T} \right)
$ 
for the same reason that their determinants agree up to sign:  their
cokernels are isomorphic to the cokernels of matrices which
are, up to sign, the transposes of each other.
\end{proof}

The proof of Theorem~\ref{thm:sign-interpretation} will ultimately
rely on the following statement about duality of {\it oriented matroids}
for vectors over an ordered field $\FF$, such as
$\FF=\QQ$.

\begin{proposition}
\label{prop:oriented-matroid}
Let $\FF$ be an ordered field, $M$ and let
$M^\perp$ be matrices in $\FF^{r \times n}$ and $\FF^{(n-r) \times n}$
as in Proposition~\ref{prop:duality}, that is, both of maximal rank, with
$\ker M$ equal to the row space of $M^\perp$.  Let the vectors $v_\ell$ in  $\FF^r$
and $v^\perp_\ell$ in $\FF^{n-r}$ be the $\ell^{th}$ columns of $M$ and $M^\perp$. 

Let $A$ be an $(r+1)$-subset of $\{1,2,\ldots,n\}$ such that
the matrix $M\big|_A$ in $\FF^{r \times (r+1)}$ has full rank $r$,
with
\begin{equation}
\label{unique-v-dependence}
\sum_{\ell \in A} c_\ell v_\ell=0
\end{equation}
the unique dependence among its columns, up to scaling. 

Then for any pair of nonzero coefficients $c_j, c_{j'} \neq 0$, 
the matrix $M^\perp\big|_{A^c\cup\{j,j'\}}$ in $\FF^{(n-r) \times (n-r+1)}$ 
has full rank $n-r$, and the unique dependence among its columns, up to scaling,
\begin{equation}
\label{unique-v-perp-dependence}
\sum_{\ell \in A^c\cup\{j,j'\}} b_{\ell} v^\perp_\ell=0,
\end{equation}
will have both $b_j, b_{j'} \neq 0$, with
$$
\frac{c_j}{c_{j'}} = - \frac{b_{j'}}{b_{j}}.
$$
In particular, $c_j,c_{j'}$ have the same sign if and only if $b_j, b_{j'}$ have opposite
signs.
\end{proposition}
\begin{proof}
First let us show the assertion about ranks.  Since $M\big|_{A}$
has rank $r$, the fact that $c_{j'} \neq 0$ implies that $v_{j'}$ lies in
the span of the columns of $M\big|_{A\setminus\{j'\}}$.  Hence the 
matrix $M\big|_{A\setminus\{j'\}}$ in $\FF^{r \times r}$ has rank $r$, and its
columns give a basis for $\FF^r$.  Then Proposition~\ref{prop:duality}(i)
implies that the columns of $M^\perp\big|_{A^c \cup \{j'\}}$ form a
basis of  $\FF^{n-r}$, and thus $M^\perp\big|_{A^c \cup \{j,j'\}}$
has rank $n-r$.  This also shows that $v^\perp_{j'}$ is
dependent on the remaining columns of $M^\perp\big|_{A^c \cup \{j,j'\}}$,
so that \eqref{unique-v-perp-dependence} will at least have $b_{j'} \neq 0$.

Now a dependence \eqref{unique-v-dependence}
extends by zeroes to a 
vector $(c_1,\ldots,c_{r+1},0,\ldots,0)$ in $\FF^n$ that 
lies in $\ker(M)$, and hence also
lies in the row space of $M^\perp$.  However, vectors in the row space
of $M^\perp$ are {\it covectors} for $\{v_\ell^\perp\}$ in the sense that
they give the values of linear functionals $f$ in $(\FF^{n-r})^*$ when
applied to the list of vectors $(v_1^\perp, \ldots,v_n^\perp)$.  
Thus there is a functional $f$ having 
$
f(v^\perp_\ell)=c_\ell
$
for $\ell \in A$
and $f(v^\perp_\ell)=0$ for $\ell \in A^c$.
Applying this $f$ to \eqref{unique-v-perp-dependence} gives
$
c_j b_j +c_{j'} b_{j'}=0
$
which is equivalent to the remaining assertion of the 
proposition.
\end{proof}

\section{Simplicial spanning trees}
\label{sec:rooting}

For a collection of subsets $S$ of some vertex set $V$, let
$\langle S \rangle$ denote the (abstract) simplicial complex $S$ on $V$
generated by $S$, that is, $\langle S \rangle \subset 2^V$ consists of
all subsets of $V$ contained in at least one subset from $S$.
We recall the notion of a simplicial spanning tree in $S$,
following Adin \cite{Adin}, 
Duval, Klivans and Martin \cite{DuvalKlivansMartin}, Kalai \cite{Kalai}, and Maxwell \cite{Maxwell}.

\begin{definition} \rm \
\label{defn:spanning-tree} 
Let $S$ be the collection of facets of a pure $k$-dimensional (abstract) simplicial
complex.  Say that $R \subset S$ is an {\it $S$-spanning tree} if
\begin{enumerate}
\item[(i)] 
$\langle R \rangle$ contains the entire $(k-1)$-skeleton of
$\langle S \rangle$, 
\item[(ii)]
$
\tilde{H}_{k}(\langle R \rangle ;\ZZ)  = 0, and
$
\item[(iii)]
$
\tilde{H}_{k-1}(\langle R \rangle ;\ZZ)
\text{ is finite.}
$
\end{enumerate}
\end{definition}

We point out here three well-known features of this definition. 
\begin{proposition}
\label{three-tree-features}
Fix the collection of facets $S$ of a pure $k$-dimensional simplicial
complex.
\begin{enumerate}
\item[(i)]
Condition (i) in Definition~\ref{defn:spanning-tree} is equivalent to
$
\tilde{H}_{k}(\langle S \rangle , \langle R \rangle;\ZZ) = \ZZ^{|S \setminus R|}.
$
\item[(ii)] Condition (ii) in Definition~\ref{defn:spanning-tree} is
equivalent to 
$
\tilde{H}_{k}(\langle R \rangle ;\QQ)  = 0.
$
\item[(iii)] All $S$-spanning trees $R$  have the same cardinality, namely
\begin{equation}
\label{tree-size}
|R|=|S| - \rank_\ZZ \tilde{H}_k(\langle S \rangle ;\ZZ).
\end{equation}
\end{enumerate}
\end{proposition}
\begin{proof}
{\sf Proof of (i).} 
Note that
$
\tilde{H}_{k}(\langle S \rangle , \langle R \rangle;\ZZ)
=\tilde{Z}_{k}(\langle S \rangle , \langle R \rangle;\ZZ),
$
since $\langle S \rangle$ is $k$-dimensional.
By definition, the relative cycle group
$\tilde{Z}_{k}(\langle S \rangle , \langle R \rangle;\ZZ)$
equals the kernel of the map
$$
\partial_k: \tilde{C}_{k}(\langle S \rangle , \langle R \rangle;\ZZ)
\rightarrow  \tilde{C}_{k-1}(\langle S \rangle , \langle R \rangle;\ZZ).
$$
Because both the source and target of $\partial_k$ are free abelian, with the source of
rank $|S \setminus R|$, its kernel is free abelian of the same rank if and only
if $\partial_k$ is the zero map.  This last condition is equivalent to
Condition (i) in Definition~\ref{defn:spanning-tree}.

\vskip .1in
\noindent
{\sf Proof of (ii).}
This follows from the fact that $\langle R \rangle$ is $k$-dimensional, so that
$$
\begin{aligned}
\tilde{H}_{k}(\langle R \rangle;\ZZ)
&=\tilde{Z}_{k}(\langle R \rangle;\ZZ), \text{ and } \\
\tilde{H}_{k}(\langle R \rangle;\QQ)
&=\tilde{Z}_{k}(\langle R \rangle;\QQ).
\end{aligned}
$$
The cycle
group $\tilde{Z}_{k}(\langle R \rangle;\ZZ)$ (respectively, 
$\tilde{Z}_{k}(\langle R \rangle;\QQ)$) vanishes if and only if
the collection of boundaries of simplices in $R$ are linearly independent
over $\ZZ$ (respectively, over $\QQ$).  However, these two notions of
linear independence are equivalent.

\vskip .1in
\noindent
{\sf Proof of (iii).}
Consider this portion of the long exact sequence of the pair 
$(\langle S \rangle , \langle R \rangle)$
$$
\begin{array}{ccccccr}
\tilde{H}_{k}(\langle R \rangle;\ZZ)
&\rightarrow
&\tilde{H}_{k}(\langle S \rangle;\ZZ)
&\rightarrow
&\tilde{H}_{k}(\langle S \rangle , \langle R \rangle;\ZZ)
&\rightarrow
&\tilde{H}_{k-1}(\langle R \rangle;\ZZ) \\
\Vert & & & &\Vert& & \\
0 & & & & \ZZ^{|S \setminus R|} & &. \\
\end{array}
$$
Here the two vertical equalities come from Condition (ii) 
in Definition~\ref{defn:spanning-tree} and from assertion (i) above, respectively.
Since the last term $\tilde{H}_{k-1}(\langle R \rangle;\ZZ)$ in this sequence 
is a finite abelian group by Condition (iii) in Definition~\ref{defn:spanning-tree},
the sequence shows that $\tilde{H}_{k}(\langle S \rangle;\ZZ)$ is a subgroup
of $\ZZ^{|S \setminus R|}$ of finite index.  Consequently, it must be free abelian
of the same rank.  Hence 
$
|S \setminus R|=\rank_\ZZ \tilde{H}_k(\langle S \rangle ;\ZZ)
$
and equation \eqref{tree-size} follows.
\end{proof}

\noindent
The following observation essentially goes back
to work of Kalai \cite[Lemma 2]{Kalai}.

\begin{proposition}
\label{prop:Kalai}
Fix a vertex set $V$ and a collection of $k$-dimensional simplices $S$.
Consider a collection of $(k+1)$-dimensional faces $T$
of cardinality 
$$
|T|:=\rank_\ZZ \tilde{H}_k(\langle S \rangle ;\ZZ)
$$
for which $T \cup \langle S \rangle$ forms a simplicial complex $K$, 
that is, all boundaries of faces in $T$ lie in $\langle S \rangle$.

Then the following two assertions hold for any choice of an $S$-spanning tree $R$.
\begin{enumerate}
\item[(i)]
The $|T| \times |T|$ matrix $\partial$ that represents the
relative simplicial boundary map 
$$
\begin{array}{ccc}
C_{k+1}(K, \langle R \rangle; \ZZ) & \rightarrow & C_k(K, \langle R \rangle; \ZZ) \\
\Vert & & \Vert \\
\ZZ^{|T|} & & \ZZ^{|S \setminus R|}
\end{array}
$$
is nonsingular if and only if $\tilde{H}_{k+1}(K; \QQ)=0$.
\item[(ii)] When the matrix $\partial$  is nonsingular, then
$
\coker(\partial) = \tilde{H}_{k}(K, \langle R \rangle; \ZZ).
$
\end{enumerate}
\end{proposition}

\begin{proof}
{\sf Proof of (i).} Start by noting that
$
\tilde{H}_k(\langle R \rangle;\QQ) =0
$
due to part (ii) of Proposition~\ref{three-tree-features}.
Since $\langle R \rangle$ is only $k$-dimensional, it also
has 
$
\tilde{H}_{k+1}(\langle R \rangle;\QQ) = 0.
$
Consequently, 
$$
\tilde{H}_{k+1}(K;\QQ)  \cong  \tilde{H}_{k+1}(K,\langle R \rangle;\QQ)= Z_{k+1}(K,\langle R \rangle;\QQ).
$$
where the isomorphism comes from
the long exact sequence in homology 
for the pair $(K,\langle R \rangle)$, and the equality
follows since $K$ is $(k+1)$-dimensional.
Therefore $\tilde{H}_{k+1}(K;\QQ)$ vanishes if and only if
$Z_{k+1}(K,\langle R \rangle;\QQ)$ vanishes, which occurs if and only
the square matrix $\partial$ has vanishing kernel, i.e. it is nonsingular.

\vskip .1in
\noindent
{\sf Proof of (ii).}
We first note that
$
C_{k}(K,\langle R \rangle;\ZZ) 
= Z_{k}(K,\langle R \rangle;\ZZ)
$
due to these facts:
\begin{enumerate}
\item[(a)] every $k$-simplex in $K$ actually lies in $S$ by our assumption
on $T$, and
\item[(b)] every boundary of a $k$-simplex in $S$ lies in $R$ by Condition (i) in
Definition~\ref{defn:spanning-tree}. 
\end{enumerate}
Thus when $\partial$ is nonsingular, one has
$$
\begin{aligned}
\coker(\partial)  
 & = C_{k}(K,\langle R \rangle;\ZZ) / B_k(K,\langle R \rangle;\ZZ) \\
 & = Z_{k}(K,\langle R \rangle;\ZZ) / B_k(K,\langle R \rangle;\ZZ) 
 = \tilde{H}_{k}(K,\langle R \rangle;\ZZ). 
\end{aligned}
$$
\end{proof}

\begin{definition} \rm \
Given a collection of $k$-simplices $S$, and an $S$-spanning tree $R$,
say\footnote{This condition on an $S$-spanning tree also plays
an important role in forthcoming work by 
Duval, Klivans and Martin \cite{DuvalKlivansMartin-personal-communication}.} 
that $R$ is {\it torsion-free} 
if Condition (iii) in Definition~\ref{defn:spanning-tree}
is strengthened to the vanishing condition 
\begin{enumerate}
\item[(iv)]
$\tilde{H}_{k-1}(\langle R \rangle ;\ZZ) = 0.$
\end{enumerate}
\end{definition}

\begin{example}
\label{ex:bouquet-example} 
\rm
For example, when $\langle R \rangle$ is a contractible subcomplex of $\langle S \rangle$
then it satisfies Condition (ii) of Definition~\ref{defn:spanning-tree}
as well as the vanishing condition (iv).  If it furthermore satisfies
Condition (i) of Definition~\ref{defn:spanning-tree}, then $R$ becomes a
torsion-free $S$-spanning tree.

A frequent combinatorial setting where this occurs (such as in Proposition~\ref{prop:shellability} 
below) is when $S$ is the set of facets of a (pure) {\it shellable} \cite{Bjorner}
simplicial complex, and $R$ is the subset of 
facets which are not fully attached along their entire boundaries
during the shelling process.
\end{example}

\begin{proposition}
\label{rooting-key-observation}
Using the hypotheses and notation of Proposition~\ref{prop:Kalai},
if one assumes in addition that $R$ is torsion-free, 
assertion (ii) of Proposition~\ref{prop:Kalai} becomes the following assertion
about (non-relative) homology:
\begin{enumerate}
\item[(ii)]
When the matrix $\partial$  is nonsingular, then
$
\coker(\partial) = \tilde{H}_{k}(K ; \ZZ)
$
\end{enumerate}
\end{proposition}
\begin{proof}
When $R$ is torsion-free, the long exact sequence for the pair $(K,\langle R \rangle)$
shows that $\tilde{H}_{k}(K ; \ZZ) \cong \tilde{H}_{k}(K , \langle R \rangle; \ZZ).$
\end{proof}

\section{More on the complete $d$-partite complex}
\label{sec:d-partite}

It is well-known and easy to see
that for a positive 
integer $n$ having prime factorization
$n=p_1^{e_1} \cdots p_d^{e_d}$ with $e_i \geq 1$, one  has 
$\Phi_n(x) = \Phi_{p_1 \cdots p_d}(
x^{n/p_1\cdots p_d})$.  Thus it suffices to interpret
the coefficients of cyclotomic polynomials for squarefree $n$.

In this section, we fix such a squarefree
$n=p_1 \cdots p_d$, and discuss further properties of the 
simplicial complexes $K_{p_1,\ldots, p_d}$, defined
in Section~\ref{sec:introduction}, appearing in Theorems~\ref{thm:second-interpretation} 
and \ref{thm:sign-interpretation}.

\begin{proposition}
\label{prop:shellability}
The $(d-2)$-dimensional skeleton of $K_{p_1,\ldots,p_d}$ is shellable, with 
$$
\tilde{H}_{d-2}(K_{p_1,\ldots,p_d};\ZZ) = \ZZ^{n - \phi(n)}.
$$
\end{proposition}
\begin{proof}
To show that the $(d-2)$-skeleton is shellable, we note the following three facts: (i) zero-dimensional complexes are all trivially shellable, (ii) joins of shellable complexes are shellable \cite[Sec. 2]{ProvanBillera}, 
and 
(iii) skeleta of (pure) shellable simplicial complexes are shellable
\cite[Corollary 10.12]{BjornerWachs}. 
Having shown that this skeleton is shellable, it therefore has only
top homology; see, for example \cite[Appendix]{Bjorner}.  This homology is free abelian, of rank
equal to the absolute value of its reduced Euler characteristic, namely
$$
\begin{aligned}
\left|\sum_{i \geq -1} (-1)^i \rank_\ZZ (C_i)\right|
 &= \left|\sum_{i \geq -1} (-1)^i \sum_{\stackrel{I \subsetneq \{1,2,\ldots, d\}}{|I| = i+1}} \prod_{i \in I} p_i \right|  \\
&= \left|\sum_{I \subsetneq \{1,2,\ldots,d\}} (-1)^{|I|-1} \prod_{i \in I} p_i\right| \\
& = \left| (p_1 - 1) \cdots (p_d - 1) - p_1 \cdots p_d \right| \\
&= |\phi(n) - n|. 
\end{aligned}
$$
\end{proof}

As noted in the introduction, the Chinese Remainder Theorem isomorphism \eqref{eqn:Chinese-remainder}
identifies elements of $\ZZ/n\ZZ$ with the $(d-1)$-dimensional simplices
of $K_{p_1,\ldots,p_d}$.  Lower dimensional faces of $K_{p_1,\ldots, p_d}$ can also be identified as cosets of subgroups within $\ZZ/n \ZZ$, but we will use this identification sparingly in this paper.  For the sake of writing down oriented simplicial boundary maps,
choose the following orientation on the simplices of $K_{p_1,\ldots,p_d}$,
consistent with the orientation of facets preceding Theorem~\ref{thm:sign-interpretation}:
choose the oriented $(\ell-1)$-simplex 
$
[j_{i_1}\bmod{p_{i_1}}, \,\, \ldots, \,\, j_{i_\ell}\bmod{p_{i_\ell}}]
$
with $i_1 < \ldots < i_\ell$ 
as a basis element of $C_{\ell-1}(K_{p_1,\ldots, p_d};\ZZ)$.
The following simple observation was the crux of the results in \cite{MartinR}.

\begin{proposition}
\label{prop:cyclotomic-and-simplicial-are-dual}
If one identifies the indexing set $\ZZ/n\ZZ$ for the columns of the boundary map
\begin{equation}
\label{eqn:top-boundary-map}
C_{d-1}(K_{p_1,\ldots,p_d};\ZZ)
\rightarrow 
C_{d-2}(K_{p_1,\ldots,p_d};\ZZ)
\end{equation}
with the set $\mu_n:=\{ \zeta^j \}_{j \in \ZZ/n\ZZ}$ of all $n^{th}$ roots of unity,
then every row of this boundary map represents a $\QQ$-linear dependence on $\mu_n$.
\end{proposition}
\begin{proof}
A row in this boundary map is indexed by an oriented $(d-2)$-face, which has 
the form
$
[j_1 \bmod p_1,\ldots, \widehat{j_k \bmod p_k},\ldots, j_d \bmod p_d] 
$
for some $j_k \in \{0, 1,\ldots, p_k-1\}$ and $1\leq k\leq d$.  
This row will then contain all zeroes except for entries equal to $(-1)^{k-1}$
in the columns indexed by $\zeta^j$ where 
$$
\begin{aligned}
j & \equiv j_1\bmod p_1,\\
  & \vdots \\
j & \equiv j_d\bmod p_d
\end{aligned}
$$
except that $j \bmod p_k$ is allowed to be arbitrary.  
These exponents $j$ are exactly those lying in one coset
of the subgroup $p_1\cdots \hat{p_k} \cdots p_d\ZZ/n\ZZ$ within $\ZZ/n\ZZ$.
Summing $\zeta^j$ over $j$ in such a coset gives zero.
\end{proof}

\begin{example} \rm
Let $n = 15$ as in Example \ref{ex:n15}, and consider the matrix for the simplicial boundary map $C_1(K_{3,5}; \ZZ) \to C_0(K_{3,5}; \ZZ)$.  One of its rows is indexed by the $0$-face $[2 \bmod 5]$ and this row has exactly three nonzero entries, all equal to $(-1)^0 = +1$.  To see these signs, we rewrite $[2 \bmod 5]$ in three ways, all of which involve deleting the first entry out of two in an oriented $1$-face: 
$$[2 \bmod 5] = [\widehat{0 \bmod 3}, 2 \bmod 5] = [\widehat{1 \bmod 3}, 2 \bmod 5]  = [\widehat{2 \bmod 3}, 2 \bmod 5].$$
The columns corresponding to these three $1$-faces are indexed by the roots of unity $\zeta^{12}$, $\zeta^7$, and $\zeta^2$, respectively.  Summing these up with coefficients of positive one, we get 
$$1\cdot\zeta^{12} + 1\cdot\zeta^7 + 1\cdot\zeta^2 = \zeta^2(\zeta^{10} + \zeta^5 + 1),$$ which is the sum of $\zeta^j$ over $j$ lying in a coset of $5\ZZ / 15\ZZ$, and hence is zero. 
\end{example}

\begin{definition} \rm
\label{defn:KT}
Assume that $n$ is squarefree and let $T$ denote any set of $n-\phi(n)$ columns of the boundary
map \eqref{eqn:top-boundary-map}.  Identify the
complementary set $T^c$ of $\phi(n)$ columns with a subset
of the $n^{th}$ roots-of-unity $\mu_n$.  
Create a subcomplex of $K_{p_1,\ldots,p_d}$ by including
its entire $(d-2)$-skeleton and attaching the subset of
$(d-1)$-faces indexed by $T$.  We denote this subcomplex as $K[T]$.
\end{definition}

Recall from the introduction that for any subset 
$A \subseteq \{0,1,\ldots,\phi(n)\} \subset \ZZ/n\ZZ$,
we let $K_A$ denote the subcomplex of $K_{p_1,\ldots,p_d}$ generated
by the facets $\{ F_{j\bmod{n}} \}$ as $j$ runs through the set of residues
$A \sqcup A_0$ 
where
$$
A_0:= \{\phi(n)+1, \,\, \phi(n)+2, \,\, \ldots, \,\, n-2, \,\, n-1\}.
$$

\begin{proposition}
\label{full-skeleton-prop}
The subcomplex $K_\varnothing$, and hence every subcomplex $K_A$, contains
the full $(d-2)$-skeleton of $K_{p_1,\ldots,p_d}$.   Consequently,
$K_A = K[A \sqcup A_0]$.
\end{proposition}
\begin{proof}
Since a $(d-2)$-face of $K_{p_1,\ldots,p_d}$ corresponds to a coset
$j_0 + \frac{n}{p_i} \left( \ZZ/n\ZZ \right)$ within $\ZZ/n\ZZ$ for some $i=1,2,\ldots,d$ and
$j_0 \in \ZZ/n\ZZ$, one must check that
every such coset intersects $A_0$.  
As $A_0$ is a consecutive sequence of $n-\phi(n)$ residues, this amounts
to checking that 
$$
n-\phi(n) \geq \frac{n}{p_i}
\quad \text{ for } \quad i=1,2,\ldots,d,
$$
or equivalently, that
$$
n\left( 1  - \frac{1}{p_i} \right) \geq \phi(n)
= n \left( 1- \frac{1}{p_1} \right) \cdots
\left( 1- \frac{1}{p_d} \right).
$$
This inequality holds since each factor in parenthesis on the right is
less than $1$.
\end{proof}

We next point out an interesting feature of the labelling of the 
boundary map \eqref{eqn:top-boundary-map} with regard to
the set $P_n$ of primitive $n^{th}$ roots of unity, noted already in \cite[Remark 5]{MartinR}.  Let $P_n^c$ denote the $(n-\phi(n))$-element 
subset of $\mu_n$ indexed by the $n^{th}$ roots of unity which are not primitive.

\begin{proposition}
\label{prop:contractible-tree}
Let $n$ be a squarefree integer and $P_n^c$ be as above.  Then the subcomplex $K[{P_n}^c]$ of $K_{p_1,\ldots, p_d}$ is contractible.
\end{proposition}
\begin{proof}
Observe that the primitive roots in $\ZZ/ n \ZZ$ are exactly those elements which do not vanish modulo $p_i$ for $i=1,\ldots, d$.  Tracing through the labelling of the $(d-1)$-faces via $\Xi$,  we obtain the description
$$
K[{P_n}^c] = \bigcup_{i=1}^d \str_{K_{p_1,\ldots,p_d}}(0 \bmod p_i), 
$$
where $\str_\Delta(v)$ denotes the {\it simplicial star} of the
vertex $v$ inside a simplicial complex $\Delta$.
Furthermore, each intersection of these stars is nonempty and
contractible, because it is the
star of another face: for $I \subset [d]$,
$$
\bigcap_{i \in I} \str_{K_{p_1,\ldots,p_d}} (0 \bmod p_i) 
\quad = \quad
\str_{K_{p_1,\ldots,p_d}} \left( \{0 \bmod p_i\}_{i \in I} \right).
$$
A standard nerve lemma \cite[Theorem 10.6]{BjornerTopMeth} then shows that
$K[{P_n}^c]$ itself is contractible.  This also follows by induction on $d$, where 
the case $d=2$ appears as \cite[Exercise 0.23, p. 20]{Hatcher}.
\end{proof}

\begin{theorem}
\label{thm:all-together}
Let $n$ be a squarefree integer and $T$ be a subset of $\mu_n$ of size $n - \phi(n)$.  Let $K[T]$ be the subcomplex of $K_{p_1,\ldots, p_d}$ of Definition~\ref{defn:KT}.  Then
$$
\tilde{H}_i(K[T];\ZZ) \cong 
\begin{cases}
 \ZZ[\zeta] / \ZZ T^c & \text{ if }i=d-2,\\
 \ZZ & \text{ if }i=d-1\text{ and }\rank_\ZZ(\ZZ T^c) = \phi(n)-1,\\
 0 & \text{ if }i=d-1\text{ and }\rank_\ZZ(\ZZ T^c) = \phi(n),\\
 0 & \text{ if }i<d-2.
\end{cases}
$$
where $\ZZ T^c$ is the sublattice $\ZZ$-spanned by the 
roots-of-unity $T^c \subset \mu_n$.
\end{theorem}
\begin{proof}
Choose any $\ZZ$-basis for $\ZZ[\zeta]$.  Let $M$ in
$\ZZ^{\phi(n) \times n}$ be the matrix that 
expresses the $n^{th}$ roots of unity $\mu_n$ in this basis.  

We construct a particular matrix $M^\perp$ to accompany $M$
as in Proposition~\ref{prop:duality} part (ii).
Consider the collection $S$ of all $(d-2)$-faces in the complete
$d$-partite complex $K_{p_1,\ldots,p_d}$.  The complex $\langle S \rangle$ 
generated by $S$ is therefore the $(d-2)$-skeleton of $K_{p_1,\ldots,p_d}$.
Proposition~\ref{prop:shellability} implies that $\langle S \rangle$ is shellable,
and that it has $\rank_\ZZ \tilde{H}_{d-2}(\langle S \rangle;\ZZ)=n - \phi(n)$.
Therefore, we are in the situation of Example \ref{ex:bouquet-example}, implying that there exists a torsion-free $S$-spanning tree $R$,
and any such $R$ will have 
$| S \setminus R |= n-\phi(n)$.

Our candidate for the matrix $M^\perp$ in $\ZZ^{(n-\phi(n)) \times n}$ is the
restriction of the boundary map from \eqref{eqn:top-boundary-map} 
to its rows indexed by $S \setminus R$.  
Proposition~\ref{prop:cyclotomic-and-simplicial-are-dual}
shows that the rows of $M^\perp$ are all perpendicular
to the rows of $M$.  

Now choose $T_0$ so that 
$T_0^c$ indexes the set $P_n$ of primitive $n^{th}$ roots of unity. 
Proposition~\ref{prop:Johnsen-precision} 
implies that the maximal minor $M\big|_{T_0^c}$ of $M$ 
is invertible over $\ZZ$, while   
Proposition~\ref{prop:contractible-tree}
implies that the maximal
minor $M^\perp\big|_{T_0}$ of $M^\perp$ 
is invertible over $\ZZ$.  Thus $M, M^\perp$ satisfy the hypotheses of
Proposition~\ref{prop:duality} part (ii), 
and combining this with Proposition~\ref{rooting-key-observation}
gives the assertion of the theorem.
\end{proof}

\begin{remark}
\label{Condro-remark}
  \rm 
  Note that Theorem~\ref{thm:all-together} makes no assertion about
  $\tilde{H}_{d-1}(K[T];\ZZ)$ when $\rank_\ZZ(\ZZ T^c) \leq \phi(n)-2$, as
  these cases will not be needed in the proofs of Theorem~\ref{thm:second-interpretation} and
  \ref{thm:sign-interpretation}.  The authors thank Giacomo Condr\`o for pointing out
  that this result was inadvertently misstated in previous versions of this paper,
  including the journal version.
\end{remark}

\section{Proof of Theorems~\ref{thm:second-interpretation}
and ~\ref{thm:sign-interpretation}}
\label{sec:all-together}

We are now in a position to prove Theorems \ref{thm:second-interpretation} and \ref{thm:sign-interpretation}.

\vskip.1in
\begin{proof}[Proof of Theorem~\ref{thm:second-interpretation}]
Let  
$T^c = \{1,\zeta,\zeta^2, \ldots, \zeta^{\phi(n)} \} \setminus \{\zeta^j\}$,
so that
$T=A_0 \cup \{j\}$, and $K[T]=K[A_0 \cup \{j\}]=K_{\{j\}}$ 
by Proposition~\ref{full-skeleton-prop}.
Note that $\rank_\ZZ(\ZZ T^c)$ is either $\phi(n)$ or $\phi(n)-1$, since
$T^c=\left( \{1,\zeta,\zeta^2, \ldots, \zeta^{\phi(n-1)} \} \cup \{ \zeta^{\phi(n)} \} \right) \setminus \{\zeta^j\}$,
and
$$
\rank_\ZZ\left( \ZZ\{1,\zeta,\zeta^2,\ldots,\zeta^{\phi(n)-1}\}\right)=\rank_\ZZ[\ZZ[\zeta]]=\phi(n).
$$
The theorem then follows from Theorem~\ref{thm:all-together} and 
Corollary~\ref{cor:first-interpretation}.
\end{proof}

\vskip.1in
\begin{proof}[Proof of Theorem~\ref{thm:sign-interpretation}]
We prove Theorem \ref{thm:sign-interpretation} by applying Proposition~\ref{prop:oriented-matroid} to the matrices $M, M^\perp$ in the proof of 
Theorem~\ref{thm:all-together}, with $A = \{1,\zeta,\zeta^2, \ldots, \zeta^{\phi(n)} \}$.  Thus $A^c =A_0$ and $K_{\{j,j'\}}=K[A^c \cup \{j,j'\}]$ by  Proposition~\ref{full-skeleton-prop}.

The dependence \eqref{unique-v-dependence} among the columns of $M|_A$
has the same coefficients (up to scaling) as the cyclotomic polynomial, and 
the dependence \eqref{unique-v-perp-dependence} among the columns of
$M^\perp\big|_{A^c\cup\{j,j'\}}$ 
will have the same coefficients (up to scaling)
as a nonzero cycle $z=\sum_{\ell} b_\ell [F_\ell]$ in $\tilde{H}_{d-1}(K_{\{j,j'\}};\ZZ)$.
\end{proof}

\section{Attaching map reformulation}
\label{sec:attachment}

The authors are indebted to Dmitry Fuchs for suggesting 
reformulations of 
Theorems~\ref{thm:second-interpretation} and \ref{thm:sign-interpretation},
explaining how the complete $d$-partite complex $K_{p_1,\ldots,p_d}$ 
is built from its subcomplex $K_\varnothing$, by attaching the facets
$F_{j \bmod{n}}$ along their boundaries.
We briefly discuss this here, beginning 
with a homological version. Throughout this section,
all homology groups are reduced, and taken with coefficients in $\ZZ$.

Recall from the Introduction that for $j$ in $\ZZ/n\ZZ$, the facet $F_{j \bmod{n}}$ was
given a particular orientation $[F_{j \bmod{n}}]$ as a basis element
in the oriented reduced $(d-1)$-chains of $K_{p_1,\ldots,p_d}$.
Let $[z_{j \bmod{n}}]:= \partial [F_{j \bmod{n}}]$ denote the $(d-2)$-cycle
which is its image under the simplicial boundary map $\partial$.

Letting $\SSS^d, \BBB^d$ denote the $d$-dimensional sphere and ball
respectively, denote by 
$\SSS^{d-1} \cup_{f} \BBB^d$ 
the space obtained by attaching $\BBB^d$ to $\SSS^{d-1}$ 
along its boundary $\Bd(\BBB^d)$ via a map $\Bd(\BBB^d) \overset{f}{\rightarrow} \SSS^{d-1}$.
Recall (see, e.g. \cite[pp. 12-13, \S 2.2, and Cor. 4.25]{Hatcher}) 
that the homotopy type of $\SSS^{d-1} \cup_{f} \BBB^d$ 
is determined by the absolute value of
$\deg(f)$, the scalar defined by the
map on the top homology groups
$$
\tilde{H}_{d-1}(\Bd(\BBB^d)) \cong \ZZ  
\overset{f_*}{\longrightarrow}
\ZZ \cong \tilde{H}_{d-1}(\SSS^d).
$$

\begin{theorem}
\label{prop:homology-attaching}
Let $n=p_1 \cdots p_d$ be squarefree.
\begin{enumerate}
\item[(i)] One has a homology isomorphism
$$
\tilde{H}_*(K_\varnothing) \cong \tilde{H}_*( \SSS^{d-2}),
$$
with $\tilde{H}_{d-2}(K_\varnothing) \cong \ZZ$ generated by
the cycle $[z_{\phi(n) \bmod{n}}]$.
\item[(ii)]
If $\Phi_n(x) = \sum_{j=0}^{\phi(n)} c_j x^j$, then
for $j=0,1,\ldots,\phi(n)$, one has
$$
[z_{j \bmod{n}}] = c_j [z_{\phi(n) \bmod{n}}]
\quad \text{ in } \quad \tilde{H}_{d-2}(K_\varnothing) \cong \ZZ,
$$
and a homology isomorphism
$$
\tilde{H}_*(K_{\{j\}}) \cong \tilde{H}_*( \BBB^{d-1} \cup_{f_j} \SSS^{d-2})
$$
where $\deg(f_j)=c_j$.
\end{enumerate}
\end{theorem}
\begin{proof}
Proposition~\ref{full-skeleton-prop} shows that
all of the spaces $K_A$ share the same $(d-2)$-skeleton as
$K_{p_1,\ldots,p_d}$, and hence they share the same homology
groups $\tilde{H}_i$ for $i < d-2$.
Furthermore, this $(d-2)$-skeleton was shown in
Proposition~\ref{prop:shellability}
to be shellable, with top cycle/homology group 
$\tilde{Z}_{d-2} \cong \ZZ^{n- \phi(n)}$.
Thus the $i$-dimensional
homology groups with $i < d-2$ for any $K_A$ will vanish,
in agreement with the homology of $\SSS^{d-1}$ and  $\SSS^{d-1} \cup_{f_j} \BBB^d$.

It only remains to show the various assertions within
(i) and (ii) for $(d-1)$- and $(d-2)$-homology.
Note that the complex $K_{\{ \phi(n) \}}$ is $\ZZ$-acyclic,
by Theorem~\ref{thm:second-interpretation} and the fact that $c_{\phi(n)}=+1$
since $\Phi_n(x)$ is monic of degree $\phi(n)$.  
Since $K_{\{ \phi(n) \}}$ also has exactly $n-\phi(n)$ facets $\{F_{j \bmod{n}}\}_{j=\phi(n)}^{n-1}$,
their boundary cycles $[z_{j \bmod{n}}]$ must form a 
$\ZZ$-basis for the 
$(d-2)$-cycle lattice 
$\tilde{Z}_{d-2} \cong \ZZ^{n- \phi(n)}$.
Since the subcomplex $K_\varnothing$ of $K_{\{\phi(n)\}}$
has the same $(d-2)$-skeleton and contains all of its facets except for
$F_{\phi(n) \bmod{n}}$, the assertions of (i)
follow.

By assertion (i), for any $j=0,1,\ldots,\phi(n)-1$, there will be
a unique integer $c$ for which 
$
[z_{j \bmod{n}}] = c [z_{\phi(n) \bmod{n}}]
$
in $\tilde{H}_{d-2}(K_\varnothing) \cong \ZZ$.
This is equivalent to the assertion that
there is a $(d-1)$-cycle in $K_{\{j,\phi(n)\}}$ of the form
$$
[F_{j \bmod{n}}] - c [F_{\phi(n) \bmod{n}}]
+\sum_{\ell=\phi(n)+1}^{n-1} b_\ell [F_{\ell \bmod{n}}].
$$
Taking $j'=\phi(n)$ in 
Theorem~\ref{thm:sign-interpretation}, and
bearing in mind that $c_{\phi(n)}=+1$,
this forces $c=c_j$, as asserted in (ii).

Lastly, note that $K_{\{j\}}$ shares the same $(d-2)$-skeleton as
$K_{\{\phi(n)\}}$, and shares most of the same facets, except for replacing
the facet $F_{\phi(n) \bmod{n}}$ with $F_{j \bmod{n}}$.  This means that
the $(d-2)$-boundaries of $K_{\{j\}}$ will span 
the sublattice of the $(d-2)$-cycle lattice $\tilde{Z}_{d-2} \cong \ZZ^{n- \phi(n)}$
in which the $\ZZ$-basis element
$
[z_{\phi(n) \bmod{n}}]
$
is replaced by
$
[z_{j \bmod{n}}] = c_j [z_{\phi(n) \bmod{n}}].
$
Thus the $(d-1)$-homology of $K_{\{j\}}$ still vanishes,
and its $(d-2)$-homology is the quotient lattice $\ZZ/c_j\ZZ$.
This agrees with $\tilde{H}_i( \BBB^{d-1} \cup_{f_j} \SSS^{d-2})$ for
$i=d-1, d-2$.
\end{proof}

Note that Theorem~\ref{prop:homology-attaching} does not
{\it circumvent} Theorems~\ref{thm:second-interpretation} and
Theorems~\ref{thm:sign-interpretation}.  It uses both
in a crucial way, thereby relying ultimately on the matroid
duality inherent in the proofs of these results.

\begin{remark} \rm
D. Fuchs also suggested the following further reformulation
of the main results.  

\begin{proposition}
Define a $(d-1)$-cochain $b$ on the complete $d$-partite
complex $K_{p_1,\ldots,p_d}$ whose value on $[F_{j \bmod{n}}]$ 
is $c_j$ for $j=0,1,\ldots,\phi(n)$, and 
$0$ otherwise. Then $b$ is a coboundary.
\end{proposition}
\begin{proof}
Extend the $\ZZ$-basis 
$\{ [z_{j \mod n}] \}_{j=\phi(n)}^{n-1}$ 
for the $(d-2)$-cycles
$\tilde{Z}_{d-2}$ to a $\ZZ$-basis 
for the $(d-2)$-chains $\tilde{C}_{d-2}$ of $K_{p_1,\ldots,p_d}$.
Then let $g$ in $GL_\ZZ(\tilde{C}_{d-2})$ send this new basis
to the standard basis of oriented $(d-2)$-faces $[f]$,
and denote by $f_0$ the $(d-2)$-face for which $[f_0]=g[z_{\phi(n) \bmod{n}}]$.
Theorem~\ref{prop:homology-attaching}(ii) shows that for 
$j=0,1,\ldots,\phi(n)$, the coefficient of $[f_0]$ 
when expanding $[z_{j \bmod{n}}]$ in the above
basis for $\tilde{Z}_{d-2}$ is $c_j$.  Thus
$g[z_{j \bmod{n}}]$ has coefficient $c_j$ on $[f_0]$.
This show that the $(d-2)$-cochain $[f_0]^*$ dual to $[f_0]$
has the property that the coboundary map $\partial^*$ sends 
$g^* [f_0]^*$ to $b$:
$$
\begin{aligned}
\partial^* g^* [f_0]^* ([F_{j \bmod{n}}]) 
&= [f_0]^* \left( g \partial[F_{j \bmod{n}}] \right) 
= [f_0]^* \left( g [z_{j \bmod{n}}] \right) \\
&= \begin{cases}
c_j & \text{ if } 0 \leq j \leq \phi(n) \\
0  & \text{ if } \phi(n)+1 \leq j \leq n-1.
\end{cases}
\end{aligned}
$$
\end{proof}

\begin{question}
Is there a natural choice of a $(d-2)$-chain having
coboundary $b$?
\end{question}
\noindent
An affirmative answer would be helpful in writing 
down the coefficients of $\Phi_n(x)$.

\end{remark}

We next give a homotopy-theoretic version of
Theorem~\ref{prop:homology-attaching}.

\begin{theorem}
For $d \geq 4$, and every  $A \subseteq \{0,1,\ldots,\phi(n)\}$,
the complex $K_A$ is simply-connected.  Consequently, for $d \neq 3$, 
one has the following.
\begin{enumerate}
\item[(i)]
The complex $K_\varnothing$ is homotopy equivalent
to $\SSS^{d-2}$, and contains $[z_{\phi(n) \bmod{n}}]$ as
a fundamental $(d-2)$-cycle.
\item[(ii)]
For $j=0,1,\ldots,\phi(n)$, the cyclotomic
polynomial coefficient $c_j$ gives the 
degree of the attaching map from the
oriented boundary $[z_{j \bmod{n}}]$ of the
facet $F_{j \bmod{n}}$ into 
the homotopy $(d-2)$-sphere $K_\varnothing$,
with respect to the choice of 
$[z_{\phi(n) \bmod{n}}]$ as
the fundamental cycle.
\item[(iii)]
In particular, the complex $K_{\{j\}}$ 
is homotopy equivalent to $\SSS^{d-2} \cup_{f_j} \BBB^{d-1}$
where $\deg(f_j)=c_j$.
\end{enumerate}
\end{theorem}
\begin{proof}
For $d=1,2$, the assertions (i),(ii),(iii) follow trivially
from Theorem~\ref{prop:homology-attaching}.  

When $d \geq 4$, first observe that
the fundamental group of $K_A$ is determined by its $2$-skeleton,
which is the same as the $2$-skeleton of  $K_{p_1,\ldots,p_d}$,
by Proposition~\ref{full-skeleton-prop}.  The latter skeleton is
shellable by Proposition~\ref{prop:shellability}, hence homotopy
equivalent to a wedge of $(d-2)$-spheres, and therefore simply-connected.

For the remaining assertions when $d\geq 4$, 
since $K_\varnothing$ is simply-connected and has the homology
of $\SSS^{d-2}$ by Theorem~\ref{prop:homology-attaching}(i),
assertion (i) follows from a standard application of the 
Hurewicz isomorphism theorem \cite[Theorem 5 part (ii), p.398]{Spanier} and 
the homological Whitehead theorem \cite[Cor. 4.33]{Hatcher}.
Assertion (ii) now follows from Theorem~\ref{prop:homology-attaching}(ii),
and assertion (iii) follows combining assertions (i) and (ii).

\end{proof}

\begin{question}  \rm
Let $d=3$, so that $n=p_1 p_2 p_3$ for three distinct primes $p_1,p_2,p_3$.
\begin{enumerate}
\item[$\bullet$]
Is $K_\varnothing$ homotopically equivalent to the circle $\SSS^1$?
\item[$\bullet$]
Is $K_{\{j\}}$ homotopically equivalent to $\BBB^2 \cup_{f_j}\SSS^1$,
where $\deg(f_j)=c_j$, for $j=0,1,\ldots,\phi(n)$?
\end{enumerate}
\end{question}

\noindent
One might hope, for example, to achieve such homotopy equivalences
by a sequence of elementary collapses.  However,
in the example of $n=105 = 3 \cdot 5 \cdot 7$, with $\phi(n)=48$,
a computer exploration indicated that 
\begin{enumerate}
\item[$\bullet$] $K_\varnothing$ did not collapse down to something
homeomorphic to $\SSS^1$, 
\item[$\bullet$]
the $\ZZ$-acyclic complexes 
$K_{\{0\}}, K_{\{48\}}$ did not collapse down to a point, and
\item[$\bullet$]
the complex $K_{\{7\}}$, which Theorem~\ref{prop:homology-attaching}
predicts has the homology of a
real projective plane $\RR P^2$ since $c_7=-2$, did not collapse down
to an $\RR P^2$.
\end{enumerate}

\section{Concordance with properties of $\Phi_n(x)$}
\label{sec:concordance}

Subtleties in the spaces $K_{\{j\}}$ make it not yet clear whether
Theorems~\ref{thm:second-interpretation} and \ref{thm:sign-interpretation} 
will prove useful
in approaching classical questions about the coefficients of $\Phi_n(x)$, e.g. as discussed in
\cite{Bachman, Elder, Moree, Kaplan, Zhao}.  Nevertheless,
we briefly explain here how various well-known properties of $\Phi_n(x)$
manifest themselves topologically in the complexes $K_{\{j\}}$ and $K_{\{j,j'\}}$
that appear in Theorems~\ref{thm:second-interpretation} and~\ref{thm:sign-interpretation}.

\subsection{Two prime factors and graphs}

When $d=2$ so that $n = p_1p_2$ is the product of only two primes, the 
subcomplexes $K_{\{j\}}$ of $K_{p_1,p_2}$ are $1$-dimensional, that is, graphs.
Hence their $(d-2)$-dimensional homology is their $0$-dimensional homology, 
which is always torsion-free.  It follows that the only nonzero 
coefficients of $\Phi_{p_1 p_2}(x)$ are $\pm 1$,
agreeing with a well-known old observation of Migotti \cite{Migotti}.
The explicit expansion of $\Phi_{p_1 p_2}(x)$ is given in
Elder \cite{Elder}, Lam and Leung \cite{LamLeung}, and Lenstra \cite{Lenstra}.

\subsection{Coefficient symmetry and simplicial automorphisms}
It is not hard to see from \eqref{eqn:recurrence-precursor} 
that, for $n > 1$, 
the cyclotomic polynomial $\Phi_n(x) = \sum_{j =0}^{\phi(n)} c_j x^j$ is palindromic,
that is, $c_j=c_{\phi(n)-j}$.  
For squarefree $n=p_1 \cdots p_d$,
this coefficient symmetry is reflected in certain simplicial automorphisms
of the complex $K_{p_1,\ldots,p_d}$ which we discuss here. 

Note that any
element of the product of symmetric groups $\symm_{p_1} \times \cdots \times \symm_{p_d}$ 
that separately permutes each of the vertex sets $K_{p_1},\ldots,K_{p_d}$ gives rise to a
simplicial automorphism of $K_{p_1,\ldots,p_d}$, with the property that
it sends a positively oriented facet $[F_{j \bmod{n}}]$ as in 
\eqref{eqn:oriented-facet} to another positively oriented facet.
We focus on a subgroup of these automorphisms isomorphic to the dihedral group
of order $2n$
$$
D_{2n}:=\langle s, r : s^2=r^n=1, srs=r^{-1} \rangle.
$$
Let $s, r$ act simultaneously on each vertex set $K_{p_i}$ as follows:
$$
\begin{aligned}
j \bmod{p_i}& \overset{s}{\longmapsto} &-j \bmod{p_i} \\
j \bmod{p_i}& \overset{r}{\longmapsto} &j+1 \bmod{p_i}. \\
\end{aligned}
$$
This induces their action on oriented facets as follows:
$$
\begin{aligned}[]
[F_{j \bmod{n}}] & \overset{s}{\longmapsto} &[F_{-j \bmod{n}}] \\
[F_{j \bmod{n}}] & \overset{r}{\longmapsto} &[F_{j+1 \bmod{n}}]. \\
\end{aligned}
$$
In particular, the element $t:=r^{\phi(n)}s$ will be an involution 
that swaps the oriented facets $[F_{j \bmod{n}}], [F_{\phi(n)-j \bmod{n}}]$
for each $j$, and hence 
maps the complex $K_{\{j\}}$ isomorphically onto the complex $K_{\{\phi(n)-j\}}$.
This shows, via Theorem~\ref{thm:second-interpretation}, that
$c_j = \pm c_{\phi(n)-j}$.

Furthermore, $t$ maps $K_{\{j,\phi(n)\}}$ isomorphically onto 
$K_{\{\phi(n)-j,0\}}$.  This shows via Theorem~\ref{thm:sign-interpretation}
that $c_j ,c_{\phi(n)-j}$ must also have the same sign:
their sign difference would have to be the same as the sign difference between the leading coefficient
$c_{\phi(n)}(=+1)$ and the constant coefficient $c_0(=+1)$.

\subsection{Cyclotomic polynomials for even $n$ and suspension}

  It is also well-known, and follows from \eqref{eqn:recurrence-precursor},
that $\phi(2n)=\phi(n)$ for $n$ odd, and that the two cyclotomic polynomials
$$
\Phi_{2n}(x) = \sum_{j=0}^{\phi(n)} \hat{c}_j x^j
\quad \text{ and } \quad
\Phi_{n}(x) = \sum_{j=0}^{\phi(n)} c_j x^j
$$
determine each other by $\Phi_{2n}(x) = \Phi_n(-x)$.  Equivalently,
$\hat{c}_j=(-1)^j c_j$ for all $j$.
When $n=p_1 \cdots p_d$ is squarefree, 
this manifests itself topologically as follows.  

The complex 
$\hat{K}:=K_{2,p_1,\ldots,p_d}$ relevant to $\Phi_{2n}(x)$
is the two-point suspension $\Sigma K$ of the complex $K:=K_{p_1,\ldots,p_d}$
relevant to $\Phi_n(x)$.  Here the two
suspension vertices in $\hat{K}$ are labelled $(0 \bmod{2})$ and
 $(1 \bmod{2})$.  
This means that every facet $F_{\ell \bmod{n}}$ of the complex $K$
is contained in exactly two facets $\hat{F}_{\ell \bmod{2n}}$ and 
$\hat{F}_{\ell+n \bmod{2n}}$ of $\hat{K}$.

Now consider the two subcomplexes $\hat{K}_{\{j\}}, K_{\{j\}}$ of $\hat{K}, K$ 
whose homology torsion predict the two coefficients $\hat{c}_j, c_j$ up to sign.
We claim that $\hat{K}_{\{j\}}$ contains the 
two-point suspension $\Sigma K_{\{j\}}$ as a deformation retract.
Specifically, $\Sigma K_{\{j\}}$ is the subcomplex of
$\hat{K}_{\{j\}}$ generated by the facets $\{ \hat{F}_{\ell \bmod{2n}} \}$
as $\ell$ runs through 
$$
\{j,n+j\} 
\quad
\cup
\quad
\{\phi(n)+1, n+\phi(n)+1,
\quad
\phi(n)+2,n+\phi(n)+2,
\quad
\ldots,
\quad
n-1,2n-1\}.
$$
The collection $\FFF$ of facets of $\hat{K}_{\{j\}}$ lying outside
the subcomplex  $\Sigma K_{\{j\}}$ is given by 
$
\{ \hat{F}_{\ell +n \bmod{2n}} \}
$
for $\ell$ in $\{0,1,\ldots,\phi(n)\} \setminus \{j\}$.
Each such facet $\hat{F}_{\ell+n \bmod{2n}}$ in $\FFF$
has a {\it free} codimension one face, 
namely the facet $F_{\ell \bmod{n}}$ of $K_{\{j\}}$,
because the facet $\hat{F}_{\ell \bmod{2n}}$ of $\hat{K}$ 
is absent from $\hat{K}_{\{j\}}$.
Thus one can remove the facets in $\FFF$ from $\hat{K}_{\{j\}}$ via a sequence
of elementary collapses, leaving the retract $\Sigma K_{\{j\}}$.
This explains why $\hat{c}_j = \pm c_j$.

The exact sign relationship $\hat{c}_j = (-1)^j c_j$
comes from a similar relationship between the
complexes $\hat{K}_{\{j,j'\}}, K_{\{j,j'\}}$ of $\hat{K}, K$:  
the two-point suspension $\Sigma K_{\{j,j'\}}$ is a deformation
retract of $\hat{K}_{\{j,j'\}}$.  One must also analyze the relation
between the orientations of facets in a $(d-2)$-cycle $z$ of $K_{\{j,j'\}}$ 
versus the orientations of the analgous facets in the suspended $(d-1)$-cycle
$\Sigma(z)$ of $\hat{K}_{\{j,j'\}}$; we omit this detailed analysis.

\section{Acknowledgements}
The authors wish to credit Jeremy Martin for the original
observation that the cyclotomic matroid $\mu_n$ with $n=p_1 p_2$ is cographic,
as well as several helpful comments.  They also thank Sam Elder, 
Nathan Kaplan, and Pieter Moree for helpful references on cyclotomic polynomials, 
Dmitry Fuchs for suggestions that led to Section~\ref{sec:attachment}, and two anonymous 
referees for helpful edits.  They are very grateful to Giacomo Condr\`o for pointing out the
error in previous versions of this paper noted in Remark~\ref{Condro-remark}.


\begin{thebibliography}{99}

\bibitem{Adin}
R.M.~Adin, 
Counting colorful multi-dimensional trees.
\textit{Combinatorica} {\bf 12} (1992), 247--260.

\bibitem{Bachman}
G. Bachman,
On the Coefficients of Cyclotomic Polynomials
{\it Memoirs of the American Mathematical Society} {\bf 106}, 
1993.

\bibitem{Bjorner}
A. Bj\"orner,
Shellable and Cohen-Macaulay partially ordered sets.  
{\it Trans. Amer. Math. Soc.}  {\bf 260}  (1980), no. 1, 159--183.

\bibitem{BjornerTopMeth}
A. Bj\"orner,
Topological methods.  
{\it Handbook of combinatorics, Vol. II}, 1819--1872.
Elsevier, Amsterdam, 1995

\bibitem{BjornerWachs}
A. Bj\"orner and M. Wachs,
Shellable nonpure complexes and posets. II.  
{\it Trans. Amer. Math. Soc.}  {\bf 349}  (1997),  no. 10, 3945--3975. 


\bibitem{Bolker}
E.D.~Bolker, 
Simplicial geometry and transportation polytopes.
\textit{Trans.\ Amer.\ Math.\ Soc.} {\bf 217} (1976), 121--142.

\bibitem{CrapoRota}
H.H.~Crapo and G.-C.~Rota,
On the foundations of combinatorial theory: Combinatorial geometries.
Preliminary edition.  The M.I.T.\ Press, Cambridge, Mass.-London, 1970.

\bibitem{DuvalKlivansMartin}
A.M. Duval, C.J. Klivans, and J.L Martin, 
Simplicial matrix-tree theorems. 
{\it Trans. Amer. Math. Soc.} {\bf 361} (2009), no. 11, 6073--6114. 

\bibitem{DuvalKlivansMartin-personal-communication}
A.M. Duval, C.J. Klivans, and J.L Martin,
personal communication, January 2011.

\bibitem{Elder}
S. Elder, Flat cyclotomic polynomials: a new approach, manuscript 
in preparation; see also 
{\tt http://www.coloradomath.org/handouts/flat\_cyclotomic.pdf}

\bibitem{Moree}
Y. Gallot and P. Moree, 
Neighboring ternary cyclotomic coefficients differ by at most one. 
{\it J. Ramanujan Math. Soc.} {\bf 24} (2009), no. 3, 235–248. 

\bibitem{Hatcher}
A. Hatcher, Algebraic Topology.
Cambridge University Press, Cambridge, 2002.

\bibitem{Johnsen}
K.~Johnsen,
Lineare Abh\"angigkeiten von Einheitswurzeln.
\textit{Elem.\ Math.} {\bf 40} (1985), 57--59.

\bibitem{Kalai}
G.~Kalai, 
Enumeration of $\QQ$-acyclic simplicial complexes.
\textit{Israel J.\ Math.} {\bf 45} (1983), 337--351.

\bibitem{Kaplan}
N. Kaplan, 
Flat cyclotomic polynomials of order three.
\textit{J. Number Theory} {\bf 127} (2007), no. 1, 118--126. 

\bibitem{Kaz}
G. S.~ Kazandzidis,
On the cyclotomic polynomial: Coefficients. 
\textit{Bull. Soc. Math. Gr\'{e}ce} (N.S.) {\bf 4} (1963), no. 1, 1--11. 

\bibitem{LamLeung}
T.Y. Lam and K.H. Leung, 
On the Cyclotomic Polynomial $\Phi_{pq}(X)$.
{\it Amer. Math. Monthly} {\bf 103} (1996), 562--564. 

\bibitem{Lenstra}
H.W. Lenstra, 
Vanishing sums of roots of unity, in: Proceedings, Bicentennial Congress Wiskundig Genootschap, Vrije Univ., Amsterdam, 1978, Part II, 1979, pp. 249--268

\bibitem{MartinR}
J.~Martin and V.~Reiner,
Cyclotomic and simplicial matroids.  
{\it Israel J. Math.}  {\bf 150}  (2005), 229--240.


\bibitem{Maxwell}
M. Maxwell,
Enumerating bases of self-dual matroids. 
{\it J. Combin. Theory Ser. A} {\bf 116} (2009), no. 2, 351--378. 

\bibitem{Meshulam}
R. Meshulam,
Homology of balanced complexes
via the Fourier transform.
{\it J. Algebraic Combin.} {\bf 35} (2012) 565--571.

\bibitem{Migotti}
A. Migotti,
Zur Theorie der Kreisteilungsgleichung.
{\it Sitzber. Math.-Naturwiss. Classe der Kaiser. Akad. der Wiss., Wien}
{\bf 87} (1883), 7--14.

\bibitem{Munkres}
J.R.~Munkres,
Elements of algebraic topology.
Addison-Wesley Publishing Company, Menlo Park, CA, 1984.

\bibitem{Oxley}
J.G.~Oxley,
Matroid theory.
Oxford Science Publications.
The Clarendon Press, Oxford University Press, New York, 1992.

\bibitem{ProvanBillera}
J.S. Provan and L.J. Billera, Decompositions of simplicial complexes related to diameters of convex polyhedra.
{\it Math. Oper. Res.} {\bf 5} (1980), no. 4, 576--594.

\bibitem{Spanier}
E.H. Spanier, 
Algebraic topology.
Springer-Verlag, New York-Berlin, 1981. 

\bibitem{Welsh}
D.J.A.~Welsh,
Matroid theory.
London Math.\ Soc.\ Monographs {\bf 8}.
Academic Press, London-New York, 1976.

\bibitem{White1}
N.~White,
Theory of matroids.
\textit{Encyclopedia of Mathematics and its Applications} {\bf 26}.
Cambridge University Press, Cambridge, 1986.

\bibitem{White2}
N.~White,
Combinatorial geometries.
\textit{Encyclopedia of Mathematics and its Applications} {\bf 29}.
Cambridge University Press, Cambridge, 1987.

\bibitem{White3}
N.~White,
Matroid applications.
\textit{Encyclopedia of Mathematics and its Applications} {\bf 40}.
Cambridge University Press, Cambridge, 1992.

\bibitem{Zhao}
J. Zhao and X. Zhang, 
Coefficients of ternary cyclotomic polynomials. 
\textit{J. Number Theory} {\bf 130} (2010), no. 10, 2223--2237.

\end{thebibliography}
\end{document}